\documentclass[11pt]{amsart}

\usepackage{amscd}
\usepackage{amsmath}
\usepackage{amssymb}
\usepackage{amsthm}
\usepackage{epsf}
\usepackage{float}
\usepackage{latexsym}
\usepackage{verbatim}
\usepackage[all, cmtip]{xy}
\usepackage{microtype}
\usepackage{tikz}
\usepackage{hyperref}
\usetikzlibrary{positioning}
\usetikzlibrary{matrix}

\newtheorem{theorem}{Theorem}[section]
\newtheorem*{unthm}{Theorem}
\newtheorem{definition}[theorem]{Definition}
\newtheorem{lemma}[theorem]{Lemma}
\newtheorem{conjecture}[theorem]{Conjecture}

\newtheorem{corollary}[theorem]{Corollary}
\newtheorem{proposition}[theorem]{Proposition}
\newtheorem*{unprop}{Proposition}

\newtheorem{varexample}[theorem]{Example}
\theoremstyle{definition}
\newtheorem{remark}[theorem]{Remark}

\newcommand{\cL}{\mathcal{L}\,}
\newcommand{\cO}{\mathcal{O} }
\newcommand{\cD}{\mathcal{D} }
\newcommand{\cE}{\mathcal{E} }

\newcommand{\bR}{\mathbb{R}}
\newcommand{\ZZ}{\mathbb{Z}}

\newcommand{\Z}{\mathbb{Z}\,}

\newcommand{\Pic}{\operatorname{Pic}}
\newcommand{\Trop}{\operatorname{Trop}}
\newcommand{\trop}{\operatorname{trop}}
\newcommand{\ddiv}{\operatorname{div}}
\newcommand{\Div}{\operatorname{Div}}
\newcommand{\PL}{\operatorname{PL}}
\newcommand{\Jac}{\operatorname{Jac}}
\newcommand{\ord}{\operatorname{ord}}
\newcommand{\an}{\operatorname{an}}

\numberwithin{equation}{section}

\newenvironment{example}{\begin{varexample}
\begin{normalfont}}{\end{normalfont}
\end{varexample}}
\begin{document}
\title[Tropicalization of Theta Characteristics]{Tropicalization of Theta Characteristics, Double Covers, and Prym Varieties}
\author{David Jensen}
\author{Yoav Len}
\date{}
\bibliographystyle{alpha}

\begin{abstract}
We study the behavior of theta characteristics on an algebraic curve under the specialization map to a tropical curve.  We show that each effective theta characteristic on the tropical curve is the specialization of $2^{g-1}$ even theta characteristics and $2^{g-1}$ odd theta characteristics.  We then study the relationship between unramified double covers of a tropical curve and its theta characteristics, and use this to define the tropical Prym variety.
\end{abstract}

\maketitle

\section{Introduction}
In this note we study the behavior of theta characteristics on an algebraic curve under the specialization map to a tropical curve.  Throughout, we let $\Gamma$ be a metric graph of genus $g$, $k$ an algebraically closed nonarchimedean field of characteristic not equal to 2, and $X$ an algebraic curve of genus $g$ over $k$ with skeleton $\Gamma$.  We denote by $\Theta (X) \subset \Pic (X)$ (respectively, $\Theta (\Gamma) \subset \Pic (\Gamma))$ the set of theta characteristics on $X$ (respectively, $\Gamma$).  Our first main result is the following.

\begin{theorem}
\label{Thm:MainThm}
\noindent \begin{enumerate}
\item (Corollary \ref{Cor:SurjTheta})  The specialization map $\Trop : \Theta (X) \to \Theta (\Gamma)$ is surjective.
\item  (Lemma \ref{Lem:NE}) The preimage of the non-effective theta characteristic $L_0 \in \Theta (\Gamma)$ consists of $2^g$ even (in fact, non-effective) theta characteristics.
\item  (Theorem \ref{Thm:HalfAndHalf}) The preimage of each other theta characteristic $L_{\gamma} \in \Theta (\Gamma)$ consists of $2^{g-1}$ even theta characteristics and $2^{g-1}$ odd theta characteristics.
\end{enumerate}
\end{theorem}
\noindent The fact that there is precisely one non-effective theta characteristic on a metric graph is due to Zharkov \cite{Zharkov10}, whose characterization of $\Theta (\Gamma)$ we review in \S \ref{Sec:Theta}.  In \cite{Panizzut16}, Pannizut states Theorem~\ref{Thm:MainThm} as a ``working hypothesis'' of Marc Coppens, and proves it in the case of hyperelliptic curves.
Our work is partially motivated by the following problem.

\begin{conjecture} \cite{BLMPR16}
\label{Conj:BLMPR}
Let $X$ be a smooth plane quartic and assume that its tropicalization $\Gamma$ is tropically smooth. Then each of the seven odd theta characteristic of $\Gamma$ is the specialization of four effective theta characteristics of $X$ (counted with multiplicity).
\end{conjecture}
\noindent
The conjecture was established by Chan and Jiradilok in \cite{ChanJiradilok15} for tropical plane quartics of a certain combinatorial type, namely $K4$-curves. Applying Theorem \ref{Thm:MainThm} to the case $g=3$ confirms Conjecture \ref{Conj:BLMPR} in its full generality.  The related problem of lifting bitangents to tropical plane curves of any degree is addressed in a recent paper by the second author and Markwig \cite{LenMarkwig}.

As observed in \cite{HarrisLen}, Theorem \ref{Thm:MainThm} extends the results of \cite{BLMPR16} to other instances of canonically embedded curves.  A canonical curve $C$ of genus 4 is the complete intersection of a quadric and a cubic in $\mathbb{P}^3$.  If the quadric is smooth, $C$ has $120$ tritangent planes in bijection with its odd theta characteristics.  If $\Gamma=\trop(C)$, then the tropicalization of each plane is tritangent to $\Gamma$.  From Theorem \ref{Thm:MainThm} it follows that the equivalence classes of the plane sections to $C$ tropicalize $8$ to $1$ to each of the $15$ odd theta characteristics on $\Gamma$.

\begin{remark}
Note that, although the parity of theta characteristics is preserved in flat families, it is not preserved by tropicalization.  Nevertheless, Theorem \ref{Thm:MainThm} shows that tropicalization of theta characteristics is in some sense well-behaved.
\end{remark}

The set $\Theta (X)$ is a torsor for the 2-torsion subgroup $\Jac_2 (X)$ of the Jacobian of $X$.  On an algebraic curve, the set $\Jac_2 (X)$ admits a non-degenerate pairing known as the Weil pairing.  Our basic strategy for proving Theorem \ref{Thm:MainThm} is to study the relationship between the pairing and the specialization map.  In particular, we show the following.

\begin{unprop}[\ref{Prop:Isotropic}]
The kernel $\Lambda$ of the specialization map
\[
\Trop: \Jac_2 (X) \to \Jac_2 (\Gamma)
\]
is isotropic for the Weil pairing.
\end{unprop}

Classically, the set of 2-torsion divisor classes on a algebraic curve is in bijection with its set of unramified double covers.  In the tropical setting, harmonic morphisms take the role of morphisms, and the correspondence between harmonic  double covers and $2$-torsion points passes through the set of cycles of the graph.  To a cycle $\gamma$, we naturally associate a divisor $D_\gamma\in\Jac_2(\Gamma)$, and to any harmonic cover $\varphi$ we associate its so-called \emph{dilation cycle} $\gamma(\varphi)$
(see~\S \ref{Sec:Covers} for  precise definitions).
As we will see, the correspondence between double covers and $2$-torsion points  is well behaved with respect to tropicalization.

\begin{unthm}[\ref{Thm:Covers}]
Let $\cD$ be a 2-torsion point in $\Jac(X)$.  Let $\overline{\varphi}: \widetilde{X} \to X$ be the corresponding double cover and $\varphi : \widetilde{\Gamma} \to \Gamma$ the specialization of this cover.  Then $\Trop (\cD) = D_{\gamma (\varphi)}$.  
\end{unthm}
\noindent
As a corollary, we give a completely combinatorial description of the full 2-torsion subgroup $\Jac_2 (X)$.  Specifically, we see that $\Jac_2 (X)$ is an extension of $\Jac_2 (\Gamma)$ by a group $\Lambda$ which is naturally identified with the set of degree 2 covering spaces of $\Gamma$.

In \S \ref{Sec:Pryms}, we define the Prym variety of an unramified degree 2 harmonic morphism.  We use this to show how the Weil pairing of certain elements of $\Jac_2 (X)$ can be computed combinatorially, using only their specializations to the graph $\Gamma$.

\subsection*{Acknowledgements}

The bulk of this paper was written during a Research in Pairs stay at Oberwolfach.  We would like to thank the institute for providing ideal working conditions for exploring these ideas.  The first author's travel was supported by an AMS Simons travel grant, and the second author was partially support by DFG grant MA 4797/6-1. We are grateful to Matt Baker for insightful remarks on a previous version of this manuscript, and thank Sam Payne, Joe Rabinoff, Dhruv Ranganathan, and Farbod Shokrieh for fielding our questions. Finally, we thank the referees for their insightful remarks.

\section{The Classical Theory of Theta Characteristics}

In this section we review the classical theory of 2-torsion points on the Jacobian of a curve, and their relation to theta characteristics, double covers, and Prym varieties.  All of the material of this section is standard, and can be found for example in \cite[Appendix B]{ACGH} or \cite{Harris82b}.

\subsection{Theta Characteristics and the Weil Pairing}

A \emph{theta characteristic} on $X$ is a divisor class $\cL$ such that $2\cL \sim K_X$.  Throughout, we denote the set of theta characteristics on $X$ by $\Theta (X)$.  The set $\Theta (X)$ is a torsor for the 2-torsion subgroup $\Jac_2 (X)$.  Since $\mathrm{char} (k) \neq 2$, $\Jac_2 (X)$ is isomorphic to $(\Z/2\Z)^{2g}$, so in particular $\Theta (X)$ has $2^{2g}$ elements.

A theta characteristic $\cL$ on $X$ is called \emph{even} if $h^0 (X,\cL)$ is even, and \emph{odd} if $h^0 (X,\cL)$ is odd.  Given a theta characteristic $\cL$, one can define a quadratic form $q_{\cL} : \Jac_2 (X) \to \Z/2\Z$ by
\[
q_{\cL} (\cD) := h^0 (X,\cL) + h^0 (X,\cL+\cD) \pmod 2 .
\]
The associated bilinear pairing
\[
\lambda (\cD,\cE) := q_{\cL} (\cD) + q_{\cL} (\cE) + q_{\cL} (\cD+\cE)
\]
is independent of $\cL$, and is known as the \emph{Weil pairing}. Recall that a subspace $\Lambda \subset \Jac_2 (X)$ is called \emph{isotropic} for a pairing $\lambda$ if $\lambda (\Lambda, \Lambda) = 0$, and similarly, it is called \emph{isotropic} for a quadratic form $q$ if $q|_\Lambda$ is identically zero.

\begin{remark}
A more standard line of exposition would be to define the Weil pairing independently, and then deduce the relation above, often called the Riemann-Mumford relation.  We, however,  take this relation as the definition, as it is all that we need for our purposes.
\end{remark}

For any non-degenerate bilinear form $\lambda$ on a vector space over $\Z/2\Z$, there are exactly 2 quadratic forms with associated bilinear form $\lambda$.  These are distinguished by the \emph{Arf invariant}.  The quadratic form with Arf invariant 1 has an isotropic subspace of dimension $g$, and has $(2^g +1)2^{g-1}$ zeros.  Conversely, the quadratic form with Arf invariant -1 does not have a $g$-dimensional isotropic subspace, and has $(2^g -1)2^{g-1}$ zeros.  One can show that the quadratic form $q_{\cL}$ has Arf invariant 1 if and only if $\cL$ is even.
By definition, $q_\cL(\cD) = 0$ precisely when the theta characteristics $\cL$ and $\cL+\cD$ have the same parity.
It follows that $\Theta (X)$ contains exactly $(2^g +1)2^{g-1}$ even theta characteristics and $(2^g -1)2^{g-1}$ odd theta characteristics.

\subsection{Double Covers and Prym Varieties}

To any 2-torsion point $\cD \in \Jac_2 (X)$, there exists a unique unramified double cover $\overline{\varphi} : \widetilde{X} \to X$ such that the kernel of the pullback map $\overline{\varphi}^* : \Jac (X) \to \Jac (\widetilde{X})$ is $\{ 0, \cD \}$.  Indeed, if $2\cD = \ddiv (f)$, then $\widetilde{X}$ is the curve with function field $K(X)(\sqrt{f})$.  Conversely, given an unramified double cover one can recover the 2-torsion point $\cD$ by considering the kernel of the pullback map.  In this way, there is a bijection between $\Jac_2 (X)$ and the set of unramified double covers of $X$.

Given such a double cover, the kernel of the pushforward map $\overline{\varphi}_* : \Jac (\widetilde{X}) \to \Jac (X)$ has two connected components.  One defines the \emph{Prym variety} $P(\overline{\varphi})$ to be the connected component of $\ker \overline{\varphi}_*$ containing 0.  The Prym variety $P(\overline{\varphi})$ is an abelian variety of dimension $g-1$.

\section{Jacobians of Metric Graphs and their Torsion Subgroups}

Recall that the divisor group $\Div (\Gamma)$ of a metric graph $\Gamma$ is the free abelian group on points of the metric space $\Gamma$.  A divisor $D = \sum a_i v_i$ on a metric graph is \emph{effective} if $a_i \geq 0$ for all $i$.  Its \emph{degree} is defined to be
\[
 \deg (D) := \sum a_i .
\]

A \emph{rational function} on a metric graph $\Gamma$ is a continuous, piecewise linear function $f:\Gamma\to\bR$ with integer slopes.  We write $\PL(\Gamma)$ for the group of rational functions on $\Gamma$.  Given $f\in\PL(\Gamma)$ and $v\in\Gamma$, we define the \emph{order of vanishing} of $f$ at $v$, denoted $\ord_v(f)$, to be the sum of the incoming slopes of $f$ at $v$.  The divisor associated to $f$ is
\[
\ddiv(f)=\sum_{v\in\Gamma}\ord_v(f)\cdot[v] .
\]
Divisors of the form $\ddiv(f)$ are called \emph{principal}.
We say that two divisors $D$ and $D'$ on a metric graph $\Gamma$ are \emph{equivalent} if $D-D'$ is principal.

The group of equivalence classes of divisors is known as the \emph{Picard group} of $\Gamma$, namely
\[
 \Pic(\Gamma)=\Div(\Gamma)/\ddiv(\PL(\Gamma)) .
\]
The \emph{Jacobian} $\Jac(\Gamma)$ of $\Gamma$ is the group of equivalence classes of divisors of degree zero. It is a $g$-dimensional real torus, so its $m$-torsion subgroup $\Jac_m (\Gamma)$ is isomorphic to $(\Z/m\Z)^g$.  The following result is a consequence of \cite{BakerRabinoff13}.

\begin{theorem}
\label{Thm:Surjective}
The specialization map on $m$-torsion subgroups
\[
\Trop : \Jac_m (X) \to \Jac_m (\Gamma)
\]
is surjective.
\end{theorem}

\begin{proof}
Let $A$ denote the value group of $k$.  Note that since $k$ is algebraically closed, the value group $A$ is divisible.  The universal cover of $\Jac(X)^{\an}$ is $(\mathbb{G}_m^{\an})^g$.  Let $M'$ be the kernel of the map $(\mathbb{G}_m^{\an})^g \to \Jac(X)^{\an}$, let $M$ be the character lattice of $\mathbb{G}_m^g$, and let $N_A = \mathrm{Hom} (M,A)$.  By \cite[(4.2.1)]{BakerRabinoff13}, there is a surjective homomorphism of short exact sequences
\[
\xymatrix{
0 \ar[r] & M' \ar[r] \ar[d]^{\cong} & \mathbb{G}_m^g \ar[r] \ar[d] & \Jac(X) \ar[r] \ar[d]^{\Trop} & 0 \\
0 \ar[r] & M' \ar[r] & N_A \ar[r] & \Sigma (\Jac(X))_A \ar[r] & 0  },
\]
where $\Sigma (\Jac(X))_A$ is defined as in \cite[Section 4]{BakerRabinoff13}.  By \cite[Theorem 1.3]{BakerRabinoff13}, $\Sigma (\Jac(X)) \cong \Jac(\Gamma)$.

The preimage of the $m$-torsion subgroup $\Jac_m (\Gamma)$ in $N_A$ is the lattice $\frac{1}{m} M'$.  Since $A$ is divisible, $N_A$ is as well, so $\frac{1}{m} M' \subset N_A$.  For any $x \in \frac{1}{m} M'$, let $y \in M' \subset \mathbb{G}_m^g$ be the preimage of $mx \in M' \subset N_A$, and let $x' \in \mathbb{G}_m^g$ be any $m$th root of $y$.  Then $x'$ maps to $x$ under the center vertical arrow, and maps to an $m$-torsion point in $\Jac(X)$.  It follows that the specialization map $\Trop : \Jac_m (X) \to \Jac_m (\Gamma)$ is surjective.
\end{proof}

Recall that by $\Theta(X)$ and $\Theta(\Gamma)$ we mean the set of theta characteristics of $X$ and $\Gamma$ respectively.
\begin{corollary}
\label{Cor:SurjTheta}
The specialization map $\Trop : \Theta (X) \to \Theta (\Gamma)$ is surjective.
\end{corollary}

\begin{proof}
Let $\cL \in \Theta (X)$ be a theta characteristic on $X$ and $L \in \Theta (\Gamma)$ be a theta characteristic on $\Gamma$.  Then $\Trop (\cL) - L \in \Jac_2 (\Gamma)$ is 2-torsion, hence by Theorem \ref{Thm:Surjective}, there exists $\cD \in \Jac_2 (X)$ such that $\Trop (\cD) = \Trop(\cL) - L$.  Then $\cL - \cD \in \Theta (X)$ is a theta characteristic on $X$ that specializes to $L$.
\end{proof}

\section{Theta Characteristics}
\label{Sec:Theta}

In this section, we consider the relationship between theta characteristics on the curve $X$ and the metric graph $\Gamma$.  The theta characteristics on a metric graph were characterized by Zharkov \cite{Zharkov10}.  We begin by describing his results.

First, there is a single non-effective theta characteristic, which can be obtained as follows.  Let $p \in \Gamma$, and consider the distance function $d_{p}$ whose value at a point $x \in \Gamma$ is the length of the shortest path from $x$ to $p$.  It is easy to see that $d_{p}$  is a piecewise linear function with all slopes of absolute value 1.
Consider the orientation $\cO$ on $\Gamma$, such that $d_p$ has only positive slopes with respect to $\cO$. For each point $x$ of $\Gamma$, denote $\mathrm{indeg}_{\cO}(x)$ the number of incoming edges at $x$ with respect to $\cO$. Then the divisor
\[
L_0 = \sum_{x \in \Gamma} (\mathrm{indeg}_{\cO}(x)-1)x
\]
has the property that $2L_0 = K_{\Gamma} + \ddiv (d_{p})$.  The class of this divisor is independent of the choice of point $p \in \Gamma$.  Since $\cO$ is acyclic, one sees by \cite[Lemma 3.2]{BakerNorine07} or \cite[Lemma 7.8]{MikhalkinZharkov08}, that the divisor $L_0$ has negative rank.  Indeed, the divisor $L_0$ is $v$-reduced and non-effective, so it is not equivalent to an effective divisor.

The remaining theta characteristics are all effective, and are in bijection with nonzero elements of $H_1 (\Gamma , \Z/2\Z)$.  Let $\gamma \in H_1 (\Gamma , \Z/2\Z)$, $\gamma \neq 0$.  By abuse of notation, we also denote by $\gamma$ a cycle on $\Gamma$ representing the class $\gamma$.  Consider the distance function $d_{\gamma}$ whose value at a point $x \in \Gamma$ is the length of the shortest path from $x$ to the cycle $\gamma$.  The function $d_{\gamma}$ is piecewise linear, with slope 0 on the cycle $\gamma$, and all other slopes of absolute value 1.  If one orients $\Gamma$ so that $\gamma$ has a totally cyclic orientation and, everywhere else, $d_{\gamma}$ has positive slopes with respect to the orientation, then one obtains an orientation $\cO$ whose associated divisor
\[
L_{\gamma} = \sum_{x \in \Gamma}(\mathrm{indeg}_{\cO}(x)-1)x
\]
has the property that $2L_{\gamma} = K_{\Gamma} + \ddiv (d_{\gamma})$.  Note that $L_{\gamma}$ is effective, and is supported on those points $x \in \Gamma$ for which the shortest distance from $x$ to $\gamma$ is obtained along at least two paths with distinct tangent directions at $x$.

Since $\Theta (\Gamma)$ contains a distinguished element $L_0$, there is a canonical bijection between $\Theta (\Gamma)$ and $\Jac_2 (\Gamma)$.  We define
\begin{equation}\label{Eq:Dilation}
D_{\gamma} := L_{\gamma} - L_0 \in \Jac_2 (\Gamma) .
\end{equation}
This is in contrast to the case of algebraic curves, where $\Theta (X)$ is a torsor for $\Jac_2 (X)$, but there is no distinguished element.

\begin{example}
\label{Ex:Theta}
We illustrate Zharkov's construction with an example.  Let $K_4$ be the complete graph on 4 vertices, with all edge lengths 1, and choose the lower left vertex to be $p$ in the construction above.
The 8 cycles on $K_4$ can be partitioned into 3 types -- the zero cycle, 4 ``triangles'', and three ``squares''.  These cycles, together with the corresponding theta characteristics, are pictured in Figure \ref{Fig:Theta}.  One can see by inspection that 7 of these theta characteristics are effective.

\begin{figure}[h]
\begin{tikzpicture}

\draw [ball color=black] (0,0) circle (0.55mm);
\draw [ball color=black] (4,0) circle (0.55mm);
\draw [ball color=black] (2,3.46) circle (0.55mm);
\draw [ball color=black] (2,1.73) circle (0.55mm);
\draw (0,0)--(4,0);
\draw (0,0)--(2,3.46);
\draw (0,0)--(2,1.73);
\draw (4,0)--(2,3.46);
\draw (4,0)--(2,1.73);
\draw (2,3.46)--(2,1.73);

\draw [ball color=black] (3,1.73) circle (0.55mm);
\draw [ball color=black] (3,0.86) circle (0.55mm);
\draw [ball color=black] (2,2.59) circle (0.55mm);
\draw (-0.3,-0.3) node {\footnotesize $-1$};
\draw [ above left] (0,0) node {\footnotesize $p$};
\draw (3.3,0.86) node {\footnotesize $1$};
\draw (2.3,2.59) node {\footnotesize $1$};
\draw (3.3,1.73) node {\footnotesize $1$};

\draw [ball color=black] (8,0) circle (0.55mm);
\draw [ball color=black] (12,0) circle (0.55mm);
\draw [ball color=black] (10,3.46) circle (0.55mm);
\draw [ball color=black] (10,1.73) circle (0.55mm);
\draw [dashed] (8,0)--(12,0);
\draw [dashed] (8,0)--(10,3.46);
\draw [dashed] (8,0)--(10,1.73);
\draw (12,0)--(10,3.46);
\draw (12,0)--(10,1.73);
\draw (10,3.46)--(10,1.73);

\draw (7.7,-0.3) node {\footnotesize $2$};

\draw [ball color=black] (4,-5) circle (0.55mm);
\draw [ball color=black] (8,-5) circle (0.55mm);
\draw [ball color=black] (6,-1.54) circle (0.55mm);
\draw [ball color=black] (6,-3.27) circle (0.55mm);
\draw [dashed] (4,-5)--(8,-5);
\draw (4,-5)--(6,-1.54);
\draw (4,-5)--(6,-3.27);
\draw (8,-5)--(6,-1.54);
\draw (8,-5)--(6,-3.27);
\draw [dashed] (6,-1.54)--(6,-3.27);

\draw [ball color=black] (6,-5) circle (0.55mm);
\draw [ball color=black] (6,-2.4) circle (0.55mm);
\draw (6,-5.3) node {\footnotesize $1$};
\draw (6.3,-2.4) node {\footnotesize $1$};

\end{tikzpicture}
\caption{The non-effective and two effective theta characteristics on $K_4$.}
\label{Fig:Theta}
\end{figure}
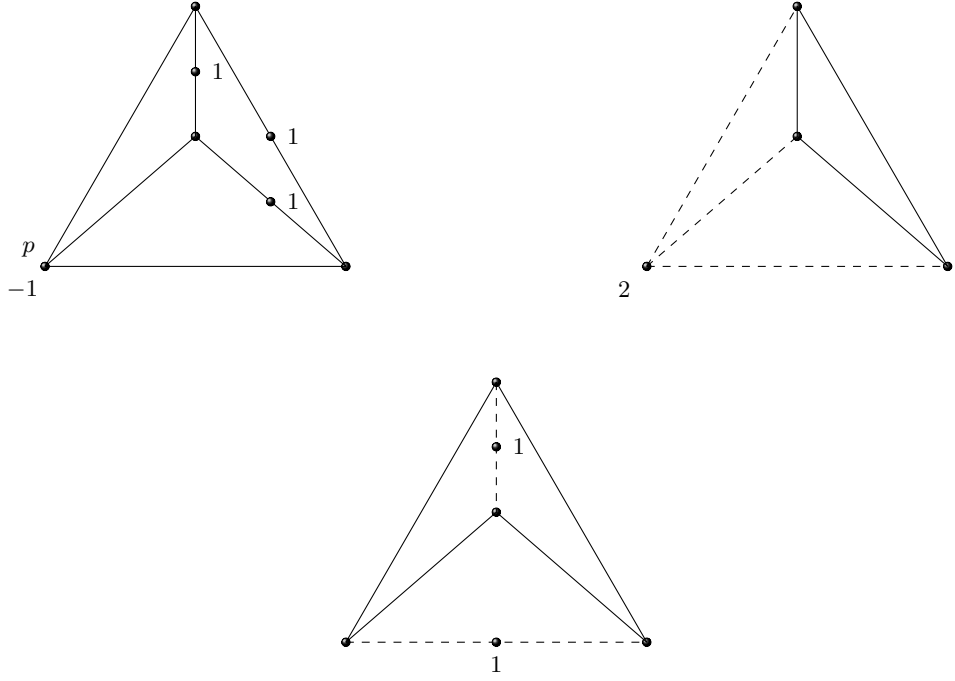

\end{example}

For the remainder of the paper, we let $\Lambda$ be the kernel of the specialization map $\Trop : \Jac_2 (X) \to \Jac_2 (\Gamma)$.  By Theorem \ref{Thm:Surjective}, there is an exact sequence
\[
0 \to \Lambda \to \Jac_2 (X) \to \Jac_2 (\Gamma) \to 0 .
\]
Thus, $\Lambda$ is a $g$-dimensional vector space over $\Z/2\Z$.  

The following is an immediate consequence of Baker's Specialization Lemma (see \cite[Lemma 2.8]{Baker08} or \cite[Theorem 1.1]{AminiBaker15}).

\begin{lemma}
\label{Lem:NE}
The preimage of $L_0$ under the specialization map $\Trop: \Theta (X) \to \Theta (\Gamma)$ consists of $2^g$ non-effective theta characteristics.
\end{lemma}

\begin{proof}
By Corollary \ref{Cor:SurjTheta}, there exists $\cL \in \Theta (X)$ such that $\Trop (\cL) = L_0$.  The preimage of $L_{\gamma}$ is then precisely $\cL + \Lambda$, which has order $2^g$.  By Baker's Specialization Lemma, $rk(\cL) \leq rk(\Trop(\cL))$.  But $\Trop(\cL) = L_0$ has rank -1, so $\cL$ has rank -1 as well.
\end{proof}

\begin{proposition}
\label{Prop:Isotropic}
The subspace $\Lambda \in \Jac_2 (X)$ is isotropic for the Weil pairing.
\end{proposition}

\begin{proof}
By Corollary \ref{Cor:SurjTheta}, there exists $\cL \in \Theta (X)$ such that $\Trop (\cL) = L_0$.  For any $\cD \in \Lambda$, by Lemma \ref{Lem:NE} we see that $h^0(X,\cL)=h^0 (X, \cL+\cD) = 0$.  Hence, for any $\cD,\cE \in \Lambda$, we have
\[
\lambda (\cD,\cE) = q_{\cL} (\cD) + q_{\cL} (\cE) + q_{\cL} (\cD+\cE) =
\]
\[
= h^0 (X,\cL+\cD) + h^0 (X,\cL+\cE) + h^0 (X,\cL+\cD+\cE) = 0.
\]
\end{proof}

\begin{corollary}
\label{Cor:Linear}
For any $\cL \in \Theta (X)$, the restriction of $q_{\cL}$ to $\Lambda$ is a linear map.
\end{corollary}

\begin{proof}
Let $\cD,\cE \in \Lambda$.  Then we have
\[
\lambda (\cD,\cE) = q_{\cL} (\cD) + q_{\cL} (\cE) + q_{\cL} (\cD+\cE) .
\]
By Proposition \ref{Prop:Isotropic}, however, the left hand side is zero.  Hence, by rearranging the terms, we obtain
\[
q_{\cL} (\cD) + q_{\cL} (\cE) = q_{\cL} (\cD+\cE) .
\]
\end{proof}

From this we deduce the main result.

\begin{theorem}
\label{Thm:HalfAndHalf}
For any effective class $L_{\gamma} \in \Theta (\Gamma)$, its preimage under the specialization map $\Trop: \Theta (X) \to \Theta (\Gamma)$ consists of $2^{g-1}$ odd theta characteristics and $2^{g-1}$ even theta characteristics.
\end{theorem}

\begin{proof}
By Corollary \ref{Cor:SurjTheta}, there exists $\cL \in \Theta (X)$ such that $\Trop (\cL) = L_{\gamma}$.  The preimage of $L_{\gamma}$ is then precisely $\cL + \Lambda$, which has order $2^g$.  By Corollary \ref{Cor:Linear}, the map $q_{\cL} : \Lambda \to \Z/2\Z$ is linear.  The kernel of this map is precisely those divisors $\cD \in \Lambda$ such that $\cL+\cD$ has the same parity as $\cL$.  Thus, if the map is trivial, then all theta characteristics in the preimage of $L_{\gamma}$ have the same parity as $\cL$, whereas if the map is surjective, then exactly half of the theta characteristics in the preimage of $L_{\gamma}$ have the same parity as $\cL$.

If $\cL$ is odd, however, then $q_{\cL}$ has Arf invariant -1, and therefore does not have any isotropic subspaces of dimension $g$.  In particular, the restriction of $q_{\cL}$ to $\Lambda$ cannot be trivial.  It follows that the preimage of $L_{\gamma}$ consists either entirely of even theta characteristics, or half even and half odd.  Recall by Lemma~\ref{Lem:NE} that in the case $\gamma = 0$, the preimage consists entirely of even theta characteristics.  Summing over all $\gamma$, we see that the number of odd theta characteristics on $X$ is $(2^g-m)2^{g-1}$, where $m$ is the number of cycles $\gamma \in H_1 (\Gamma , \Z/2\Z)$ such that the preimage of $L_{\gamma}$ contains only even theta characteristics.  It is well known, however, that an algebraic curve possesses exactly $(2^g-1)2^{g-1}$ odd theta characteristics, hence $m=1$.
\end{proof}

\begin{example}
\label{Ex:Bitangents}
Returning to Example \ref{Ex:Theta}, the 28 odd theta characteristics on a non-hyperelliptic genus 3 curve are the points of intersection of the bitangent lines to the curve in its canonical embedding.  By Theorem \ref{Thm:HalfAndHalf}, each of the 7 effective theta characteristics on $K_4$ pictured in Figure \ref{Fig:Theta} is the specialization of exactly 4 odd theta characteristics on $X$.  In this example, the 7 effective theta characteristics on $K_4$ are rigid in the sense that they each have only one effective divisor in their class, so in fact the points of intersection of the 28 bitangents specialize 4-to-1 to precisely these \emph{divisors} on $K_4$, not just these divisor \emph{classes}.
\end{example}

\section{Unramified double Covers}
\label{Sec:Covers}

Recall that there is a one-to-one correspondence between unramified double covers of an algebraic curve $X$ and 2-torsion points in $\Jac (X)$.  We now explore the relationship between unramified double covers of a metric graph and 2-torsion points in its Jacobian.  We first review the necessary material on augmented metric graphs and harmonic morphisms.  For a more detailed discussion of harmonic morphisms, see \cite{ABBR1}.

\begin{definition}
An \emph{augmented metric graph} is a metric graph $\Gamma$, together with a function $g : \Gamma \to \Z_{\geq 0}$, called the \emph{genus function}, such that $g(x) = 0$ for all but finitely many $x \in \Gamma$.
\end{definition}

We often neglect the weight function if it is clear from context.
Any metric graph may be seen as an augmented graph for which the genus function is identically zero.  To emphasize that,  we refer to them as \emph{unaugmented} metric graphs.
To an augmented metric graph $(\Gamma, g)$ we associate an unaugmented metric graph $\Gamma^{\#}$ by adding $g(x)$ loops of  length $\epsilon>0$ based at $x$ for each point $x \in \Gamma$, and forgetting the weight. The Jacobian of  $(\Gamma, g)$ is, by definition,  the Jacobian of $\Gamma^{\#}$. Note that none of our results below depends on the lengths of these loops.


\begin{definition}
A continuous map $\varphi : \widetilde{\Gamma} \to \Gamma$ of metric graphs is called a \emph{morphism} if there exist vertex sets $\widetilde{V} \subset \widetilde{\Gamma}, V \subset \Gamma$, such that $\varphi (\widetilde{V}) \subseteq V$, $\varphi^{-1} E(\Gamma) \subseteq E(\widetilde{\Gamma})$, and the restriction of $\varphi$ to any edge $\tilde{e}$ of $\widetilde{\Gamma}$ is a  dilation by some factor $d_{\tilde{e}} (\varphi) \in \Z_{\geq 0}$.  A morphism is called \emph{finite} if $d_{\tilde{e}} (\varphi) > 0$ for all edges $\tilde{e}$.
\end{definition}

The tropical analogue of a map between algebraic curves is not just a morphism, but a harmonic morphism. For a point $x$ of a metric graph $\Gamma$, we denote $T_x(\Gamma)$ the collection of its tangent directions (which is in bijection with the edges emanating from $x$).

\begin{definition} \label{Def:harmonic}
A finite morphism of metric graphs is called \emph{harmonic} at $\tilde{x} \in \widetilde{\Gamma}$ if the sum
\[
d_{\tilde{x}} (\varphi) := \sum_{\tilde{v} \in T_{\tilde{x}} (\widetilde{\Gamma}), \varphi (\tilde{v}) = v } d_{\tilde{v}} (\varphi)
\]
is independent of the choice of tangent vector $v \in T_{\varphi ({\tilde{x}})} (\Gamma)$.  The number $d_{\tilde{x}} (\varphi)$ is called the \emph{degree} of the harmonic morphism at $\tilde{x}$.
The morphism is called \emph{harmonic} if it is surjective and harmonic at every point $\tilde{x} \in \widetilde{\Gamma}$.  In this case, the number $\deg(\varphi) = \sum_{\varphi (\tilde{x})= x} d_{\tilde{x}}(\varphi)$ is independent of $x$, and is called the \emph{degree} of the harmonic morphism $\varphi$.
\end{definition}

Our primary interest here is in unramified harmonic morphisms.  Let $(\widetilde{\Gamma},\tilde{g})$ and $(\Gamma,g)$ be augmented metric graphs.
Given a point $\tilde{x} \in \widetilde{\Gamma}$, the ramification of $\varphi$ at $\tilde{x}$ is
\[
R_{\tilde{x}} = (2-2\tilde{g}(\tilde{x})) - d_{\tilde{x}} (\varphi) (2-2g(\varphi(\tilde{x}))) - \sum_{\tilde{v} \in T_{\tilde{x}} (\widetilde{\Gamma})} (d_{\tilde{v}} (\varphi) -1),
\]
or equivalently, $R_{\tilde{x}} = K_{\widetilde{\Gamma}}(\tilde{x}) - \varphi^*(K_\Gamma)(\tilde{x})$, where $K_{\widetilde{\Gamma}}$ and $K_\Gamma$ are the canonical divisors of $\widetilde{\Gamma}$ and $\Gamma$ respectively.
We say that the harmonic morphism $\varphi$ is \emph{unramified} if $R_{\tilde{x}} = 0$ for all $\tilde{x} \in \widetilde{\Gamma}$.

We often refer to an unramified harmonic morphism of degree 2 as an \emph{unramified double cover}. Throughout this section, we are only interested in such covers in which the target is unaugmented. However, we cannot rule out the possibility that the genus function of the covering graph $\widetilde{\Gamma}$ is non-trivial.
Given an unramified double cover $\varphi : \widetilde{\Gamma} \to \Gamma$, we define the \emph{dilation cycle} $\gamma (\varphi)$ to be the set of points of $\Gamma$ in the image of a dilated edge.  More precisely,
\[
\gamma (\varphi) : = \{ x \in \Gamma \vert \exists \tilde{x} \in \widetilde{\Gamma}, \tilde{v} \in T_{\tilde{x}} (\widetilde{\Gamma}) \text{ with } \varphi (\tilde{x}) = x \text{ and } d_{\tilde{v}} (\varphi) = 2 \}.
\]
We now describe the genus function of an unramified double cover.

\begin{lemma}
\label{Lem:Genus}
Let $\varphi : \widetilde{\Gamma} \to \Gamma$ be an unramified double cover, with $\Gamma$  unaugmented.
\begin{enumerate}
\item  If $x \notin \gamma (\varphi)$, then $\varphi^{-1} (x)$ consists of two points of genus 0.
\item  If $x \in \gamma (\varphi)$, then $\varphi^{-1} (x)$ is a single point of genus $\frac{1}{2} \deg_{\gamma (\varphi)} (x) - 1$, where $\deg_{\gamma (\varphi)} (x)$ is the number of tangent vectors at $x$ contained in $\gamma (\varphi)$.
\end{enumerate}
\end{lemma}

\begin{proof}
Since $\varphi$ is unramified and the genus function of $\Gamma$ is trivial, for any point $\tilde{x} \in \widetilde{\Gamma}$ the ramification formula gives us
\[
2g(\tilde{x})+2 = \sum_{\tilde{v} \in T_{\tilde{x}} (\widetilde{\Gamma})} (d_{\tilde{v}} (\varphi) -1)
\]
if $d_{\tilde{x}} = 2$, and
\[
2g(\tilde{x}) = \sum_{\tilde{v} \in T_{\tilde{x}} (\widetilde{\Gamma})} (d_{\tilde{v}} (\varphi) -1)
\]
if $d_{\tilde{x}} = 1$.

Note that if $d_{\tilde{x}} = 2$ and $\varphi (\tilde{x})$ is not in the dilation cycle, then this implies that $g(\tilde{x}) < 0$, which is impossible.  It follows that $d_{\tilde{x}} = 2$ if and only if $\varphi (\tilde{x})$ is in the dilation cycle.  As $\varphi$ is finite and degree 2, each dilation factor $d_{\tilde{v}} (\varphi)$ is either 1 or 2. The right hand side of the expression above is therefore equal to $\deg_{\gamma (\varphi)} (x)$, and the result follows.
\end{proof}

The following corollary explains our choice of terminology.

\begin{corollary}
The dilation cycle $\gamma (\varphi)$ is a cycle in $\Gamma$.
\end{corollary}

\begin{proof}
By Lemma \ref{Lem:Genus}, for every point $x \in \Gamma$ there is an even number of tangent vectors in $T_x (\Gamma)$ that are contained in $\gamma (\varphi)$.  It follows that $\gamma (\varphi)$ is a cycle.
\end{proof}

\begin{remark}
In many applications, one is primarily interested in trivalent metric graphs, since these form an open dense subset of the moduli space of tropical curves.  By Lemma \ref{Lem:Genus}, we see that if $\Gamma$ is trivalent, then the genus function on $\tilde{\Gamma}$ is identically zero.  In this case, there is no ambiguity with how one defines the Jacobian of the cover $\widetilde{\Gamma}$.
\end{remark}

Our first result on double covers concerns the kernel of the pullback map on divisors.  Recall that the Jacobian of an augmented graph is defined to be the Jacobian of the unaugmented graph obtained by introducing $g(x)$ loops of length $\epsilon$ based at $x$ for every point $x$ in the graph.  Note that, if $\varphi : \widetilde{\Gamma} \to \Gamma$ is a harmonic morphism with $\Gamma$ unaugmented, then the pullback and pushforward maps on divisors are well-defined.  Specifically, points in the small loops of $\widetilde{\Gamma}^{\#}$ pushforward to the same points as the basepoint of the loop, and points in $\Gamma$ always pullback to the basepoints of the small loops.

\begin{proposition}
\label{Prop:Kernel}
Let $\varphi : \widetilde{\Gamma} \to \Gamma$ an unramified double cover with $\Gamma$ unaugmented.
Then
\[
\ker \varphi^* = \{ D_{\gamma} \mid \gamma\subseteq\gamma(\varphi)\},
\]
where $D_\gamma$ is the divisor associated to the cycle $\gamma$ as defined in \eqref{Eq:Dilation}.
\end{proposition}

\begin{proof}
First, note that the composition $\varphi_* \varphi^*$ is simply multiplication by 2, so the kernel of $\varphi^*$ must be contained in the set of 2-torsion points. Let $d_{p}$ and $d_{\gamma}$ be as defined in the beginning of \S \ref{Sec:Theta}.

As a first step, observe that for a cycle $\gamma$, the divisor $\varphi^* D_{\gamma}$ is principal if and only if the slope of the piecewise linear function $\varphi^* (d_{\gamma} - d_{p})$ is even at every tangent vector in $\widetilde{\Gamma}$.  Indeed, assume that the pullback $\varphi^*(D_\gamma)$ is the principal divisor associated to a piecewise linear function $f$. By the fact that $2D_\gamma = \ddiv (d_{\gamma} - d_{p})$, we have
 $\ddiv (\varphi^* (d_{\gamma} - d_{p})) = \ddiv (2f)$.  It follows that $2f-(d_{\gamma} - d_{p})$ must be a constant.  In other words, up to translation by a constant, we have
\[
 f = \frac{1}{2} \varphi^* (d_{\gamma} - d_{p}),
\]
which is well-defined if and only if all of the slopes of $\varphi^* (d_{\gamma} - d_{p})$ are even.

Now, since $d_p$ has odd slopes at every tangent vector, and the slope of  $d_\gamma$ is even precisely on tangent vectors in $\gamma$, we see that the slope of $d_{\gamma} - d_{p}$ is \emph{odd} preciesly at  tangent vectors in $\gamma$.
Pulling back via $\varphi^*$, the slope is rescaled by $d_{v}$, so  the slope of $\varphi^* (d_{\gamma} - d_{p})$ at a tangent vector $v$ in $\widetilde{\Gamma}$  is odd preciesly at tangent vectors in $\varphi^{-1}(\gamma \smallsetminus \gamma(\varphi))$. We conclude that the slopes of $\varphi^* (d_{\gamma} - d_{p})$ are everywhere even (and therefore $D_\gamma$ is principal) if and only if $\gamma \subseteq \gamma (\varphi)$.
\end{proof}

Note in particular that there exist unramified double covers $\varphi : \widetilde{\Gamma} \to \Gamma$ for which $\varphi^*$ is injective.  These are precisely the covers for which there is no dilation, that is $\gamma (\varphi) = 0$.  In other words, these are the covering spaces of $\Gamma$ of degree $2$.  A graph of genus $g$ admits exactly $2^g$ covering spaces of degree 2, which can be described as follows.  Choose a spanning tree $T$ of $G$.  Any degree 2 covering space will contain a disjoint union $T \sqcup T'$ of two copies of $T$.  Now, for each of the $g$ edges in the complement of $T$, there are 2 possible ways to lift it to the double cover.  If the edge connects vertices $v_1$ to $v_2$, we may lift it to two edges connecting $v_1$ to $v_2$ and $v_1'$ to $v_2'$, or to two edges connecting $v_1$ to $v_2'$ and $v_1'$ to $v_2$.  To see that this is the complete set of degree 2 covering spaces of $\Gamma$, note that the fundamental group of $\Gamma$ is invariant under homotopy, and thus such covering spaces are in bijection with degree 2 covering spaces of the rose obtained by contracting $T$.  An exposition of this can be found in \cite{Waller76}.


For an unramified double cover we have the following.

\begin{lemma}
\label{Lem:harmonicCovers}
Let $\gamma$ be a cycle in an unaugmented metric graph $\Gamma$, and let $h$ be the sum of the genera of the connected components of $\Gamma \smallsetminus \gamma$.  Then up to isomorphism, there are $2^h$ unramified double covers $\varphi : \widetilde{\Gamma}\to\Gamma$ with dilation cycle $\gamma$.
\end{lemma}

\begin{proof}
This follows immediately from the fact that the preimage of each connected component of $\Gamma \smallsetminus \gamma$ is a covering space of degree $2$.
\end{proof}

For the rest of the section, fix a curve $X$ and its skeleton $\Gamma$.
We now consider the relationship between unramified double covers of the graph $\Gamma$ and the curve $X$.  Any unramified double cover $\overline{\varphi} : \widetilde{X} \to X$ descends to an unramified double cover of skeletons.  The following proposition establishes the converse.

\begin{lemma}
\label{Lem:Hurwitz}
Let $\gamma$ be a cycle in an unaugmented metric graph $\Gamma$ and let $h$ be the sum of the genera of the connected components of $\Gamma \smallsetminus \gamma$.  Let $\varphi : \widetilde{\Gamma} \to \Gamma$ be an unramified double cover with dilation cycle $\gamma$.  Then there are $2^{g-h}$ unramified double covers $\overline{\varphi}: \widetilde{X} \to X$ that specialize to $\varphi$.
\end{lemma}

\begin{proof}
Let $V$ be the vertex set consisting of all points of $\Gamma$ of valence different from 2.  Let $G$ be the corresponding model, and let $G_1 , \ldots , G_m$ be the connected components of $G \smallsetminus \gamma$.  We first show that the statement holds in the case that $G$ is 3-regular.  In this case, $\gamma$ is a union of simple cycles, and every vertex of $\gamma$ is contained in precisely one of the connected components $G_i$.  Therefore,
\[
\vert V(G) \vert = \sum_{i=1}^m \vert V(G_i) \vert,
\]
and it follows that
\[
\vert E(\gamma) \vert = \vert E(G) \vert - \sum_{i=1}^m \vert E(G_i) \vert
\]
\[
= g + \vert V(G) \vert -1 + \sum_{i=1}^m \left( g(G_i) + \vert V(G_i) \vert -1 \right) = g+m-h-1 .
\]
The cover $\varphi : \widetilde{\Gamma} \to \Gamma$ has $2^m$ automorphisms, given by involution on the preimages of the connected components of $\Gamma \smallsetminus \gamma$.  In this case, the local Hurwitz number at every point $v \in V$ is 1.  Hence, by \cite[Theorem 2]{CMR}, there are
$2\frac{2^{g+m-h-1}}{2^m} = 2^{g-h}$ double covers $\overline{\varphi}: \widetilde{X} \to X$ that specialize to $\varphi$.

We now show that this number does not change when  contracting a non-loop edge.  Since every graph of genus $g$ can be obtained from a 3-regular graph by contracting edges, the result will follow.  First, if we contract an edge that does not intersect $\gamma$, then the calculation above does not change.  Next, if we contract an edge that is not contained in $\gamma$ but with both endpoints contained in $\gamma$, then the local Hurwitz numbers multiply and the number of automorphisms does not change.  Third, if we contract an edge $e$ contained in $\gamma$, then the local Hurwitz number at the image of $e$  equals twice the product of the local Hurwitz numbers at two endpoints of $e$.  Again, the calculation above does not change.

Finally, if we contract an edge $e$ with one endpoint $v$ in $\gamma$ and the other endpoint $w$ not in $\gamma$, then the local Hurwitz number at $v$ increases by a factor of $2^{\textrm{val}(w)-2}$.  If $G_i$ is the connected component of $G \smallsetminus \gamma$ containing $e$, and $G'_1 , \ldots , G'_k$ are the connected components of its image under the contraction, then $h$ increases by $k+1 - \textrm{val}(w)$ and the number of automorphisms of the cover increases by a factor of $2^{k-1}$.  Thus, the number of double covers specializing to $\varphi$ does not change.
\end{proof}

An immediate consequence of Lemmas \ref{Lem:harmonicCovers} and \ref{Lem:Hurwitz} is the following.

\begin{proposition}
\label{Prop:Lifting}
Let $\gamma$ be a cycle in an unaugmented metric graph $\Gamma$.  Then there are $2^g$ unramified double covers $\overline{\varphi} : \widetilde{X} \to X$  specializing to an unramified double cover with dilation cycle $\gamma$.
 \end{proposition}

\begin{proof}
By Lemma \ref{Lem:harmonicCovers}, there are $2^h$ unramified double covers with dilation cycle $\gamma$.  By Lemma \ref{Lem:Hurwitz}, each of these lifts to $2^{g-h}$ double covers of $X$.  Hence there are $2^h \cdot 2^{g-h} = 2^g$ such double covers of $X$.
\end{proof}

We now describe the connection between covers of $\Gamma$ and covers of $X$.

\begin{theorem}
\label{Thm:Covers}
Let $\cD$ be a 2-torsion point in $\Jac(X)$.  Let $\overline{\varphi}: \widetilde{X} \to X$ be the corresponding double cover and $\varphi : \widetilde{\Gamma} \to \Gamma$ the specialization of this cover.  Then $\Trop (\cD) = D_{\gamma (\varphi)}$.
\end{theorem}

\begin{proof}
Consider the following diagram.
\[
\xymatrix{
\Jac (X) \ar[r]^{\Trop} \ar[d]^{\overline{\varphi}^*}  & \Jac (\Gamma) \ar[d]^{\varphi^*} \\
\Jac (\widetilde{X}) \ar[r]^{\Trop} & \Jac (\widetilde{\Gamma}) }
\]
Since $\overline{\varphi}^* \cD = 0$, we see that $\varphi^* \Trop(\cD) = 0$ as well.  By Proposition \ref{Prop:Kernel}, it follows that $\Trop(\cD) = D_\gamma$ for some $\gamma\subseteq\gamma(\varphi)$.

Suppose by induction that the theorem is true for every cycle properly contained in $\gamma(\varphi)$. That is, for any cycle $\gamma' \subsetneq \gamma(\varphi)$,
every double cover $\overline\psi$ of $X$ that specializes to an unramified double cover $\psi$ of $\Gamma$ with dilation cycle $\gamma'$ corresponds to a divisor $D_{\overline{\psi}}$ on $X$ that specializes to $D_{\gamma'}$.  For any such $\gamma'$, there are $2^g$ divisor classes specializing to $D_{\gamma'}$ (since this is the order of the kernel $\Lambda$), and by Proposition \ref{Prop:Lifting} there are $2^g$ double covers of $X$ specializing to an unramified double cover of $\Gamma$ with dilation cycle $\gamma'$.  Because these two sets have the same size, every divisor specializing to $D_{\gamma'}$ corresponds to a double cover that specializes to an unramified double cover with dilation cycle $\gamma'$.  It follows that the containment $\gamma \subseteq \gamma (\varphi)$ cannot be strict.
\end{proof}

\begin{remark}
By Theorem~\ref{Thm:Covers}, there is a canonical bijection between  $\Lambda$ (the kernel of the map $\Jac_2 (X) \to \Jac_2 (\Gamma)$) and the set of degree 2 honest covering spaces of $\Gamma$. Because of this, the combinatorics of the graph $\Gamma$ ``sees'' the entire 2-torsion subgroup $\Jac_2 (X)$.
More precisely, $\Jac_2 (X)$ is an extension of $\Jac_2 (\Gamma) \cong H_1 (\Gamma, \Z/2\Z)$ by the set of covering spaces of $\Gamma$. In fact, this extension splits (non-canonically) as all the groups involved are vector spaces.  \end{remark}

\begin{example}
\label{Ex:Cover}
As an example,  consider the honest double cover of our earlier example $K_4$ by the cube $\widetilde{K_4}$, again with all edge lengths 1, pictured in Figure \ref{Fig:Cube}.  By Lemma \ref{Lem:Hurwitz}, there exists a unique double cover $\overline{\varphi} : \widetilde{X} \to X$ specializing to this double cover of $K_4$.

\begin{figure}[h]
\begin{tikzpicture}

\draw [ball color=black] (0,0) circle (0.55mm);
\draw [ball color=black] (4,0) circle (0.55mm);
\draw [ball color=black] (0,4) circle (0.55mm);
\draw [ball color=black] (4,4) circle (0.55mm);
\draw [ball color=black] (1,1) circle (0.55mm);
\draw [ball color=black] (1,3) circle (0.55mm);
\draw [ball color=black] (3,1) circle (0.55mm);
\draw [ball color=black] (3,3) circle (0.55mm);
\draw (0,0)--(4,0);
\draw (0,0)--(0,4);
\draw (0,0)--(1,1);
\draw (4,0)--(4,4);
\draw (4,0)--(3,1);
\draw (0,4)--(4,4);
\draw (0,4)--(1,3);
\draw (4,4)--(3,3);
\draw (1,1)--(1,3);
\draw (1,1)--(3,1);
\draw (1,3)--(3,3);
\draw (3,1)--(3,3);
\draw (-0.3,-0.3) node {\footnotesize $A$};
\draw (-0.3,4.3) node {\footnotesize $B$};
\draw (4.3,-0.3) node {\footnotesize $C$};
\draw (1.3,1.3) node {\footnotesize $D$};
\draw (4.3,4.3) node {\footnotesize $D'$};
\draw (1.3,2.7) node {\footnotesize $C'$};
\draw (2.7,1.3) node {\footnotesize $B'$};
\draw (2.7,2.7) node {\footnotesize $A'$};

\draw [->] (5,2)--(7,2);

\draw [ball color=black] (8,0) circle (0.55mm);
\draw [ball color=black] (12,0) circle (0.55mm);
\draw [ball color=black] (10,3.46) circle (0.55mm);
\draw [ball color=black] (10,1.73) circle (0.55mm);
\draw (8,0)--(12,0);
\draw (8,0)--(10,3.46);
\draw (8,0)--(10,1.73);
\draw (12,0)--(10,3.46);
\draw (12,0)--(10,1.73);
\draw (10,3.46)--(10,1.73);
\draw (7.7,-0.3) node {\footnotesize $A$};
\draw (12.3,-0.3) node {\footnotesize $C$};
\draw (10,4.03) node {\footnotesize $B$};
\draw (10,1.43) node {\footnotesize $D$};

\end{tikzpicture}
\caption{A double cover of $K_4$ by a cube.}
\label{Fig:Cube}
\end{figure}
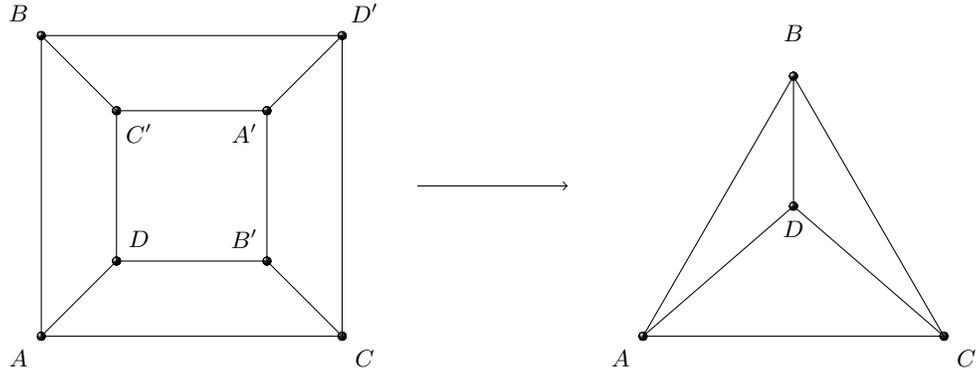

Figure \ref{Fig:Pullback} depicts the pullback of $D_{\gamma}$ for each of the two effective theta characteristics $L_{\gamma}$ pictured in Figure \ref{Fig:Theta}.  The divisor on the left corresponds to the case where $\gamma$ is a triangle, whereas the divisor on the right corresponds to the square.  We leave it to the reader to verify that neither of these divisors is equivalent to zero.

\begin{figure}[h]
\begin{tikzpicture}

\draw [ball color=black] (0,0) circle (0.55mm);
\draw [ball color=black] (4,0) circle (0.55mm);
\draw [ball color=black] (0,4) circle (0.55mm);
\draw [ball color=black] (4,4) circle (0.55mm);
\draw [ball color=black] (1,1) circle (0.55mm);
\draw [ball color=black] (1,3) circle (0.55mm);
\draw [ball color=black] (3,1) circle (0.55mm);
\draw [ball color=black] (3,3) circle (0.55mm);
\draw [ball color=black] (2,4) circle (0.55mm);
\draw [ball color=black] (4,2) circle (0.55mm);
\draw [ball color=black] (1,2) circle (0.55mm);
\draw [ball color=black] (2,1) circle (0.55mm);
\draw [ball color=black] (0.5,3.5) circle (0.55mm);
\draw [ball color=black] (3.5,0.5) circle (0.55mm);
\draw (0,0)--(4,0);
\draw (0,0)--(0,4);
\draw (0,0)--(1,1);
\draw (4,0)--(4,4);
\draw (4,0)--(3,1);
\draw (0,4)--(4,4);
\draw (0,4)--(1,3);
\draw (4,4)--(3,3);
\draw (1,1)--(1,3);
\draw (1,1)--(3,1);
\draw (1,3)--(3,3);
\draw (3,1)--(3,3);
\draw (-0.3,-0.3) node {\footnotesize $3$};
\draw (2.7,2.7) node {\footnotesize $3$};
\draw (2,4.3) node {\footnotesize $-1$};
\draw (4.3,2) node {\footnotesize $-1$};
\draw (1.3,2) node {\footnotesize $-1$};
\draw (2,1.3) node {\footnotesize $-1$};
\draw (3.5,0.2) node {\footnotesize $-1$};
\draw (0.5,3.75) node {\footnotesize $-1$};

\draw [ball color=black] (8,0) circle (0.55mm);
\draw [ball color=black] (12,0) circle (0.55mm);
\draw [ball color=black] (8,4) circle (0.55mm);
\draw [ball color=black] (12,4) circle (0.55mm);
\draw [ball color=black] (9,1) circle (0.55mm);
\draw [ball color=black] (9,3) circle (0.55mm);
\draw [ball color=black] (11,1) circle (0.55mm);
\draw [ball color=black] (11,3) circle (0.55mm);
\draw [ball color=black] (10,0) circle (0.55mm);
\draw [ball color=black] (12,2) circle (0.55mm);
\draw [ball color=black] (9,2) circle (0.55mm);
\draw [ball color=black] (10,3) circle (0.55mm);
\draw [ball color=black] (8.5,3.5) circle (0.55mm);
\draw [ball color=black] (11.5,0.5) circle (0.55mm);
\draw (8,0)--(12,0);
\draw (8,0)--(8,4);
\draw (8,0)--(9,1);
\draw (12,0)--(12,4);
\draw (12,0)--(11,1);
\draw (8,4)--(12,4);
\draw (8,4)--(9,3);
\draw (12,4)--(11,3);
\draw (9,1)--(9,3);
\draw (9,1)--(11,1);
\draw (9,3)--(11,3);
\draw (11,1)--(11,3);
\draw (7.7,-0.3) node {\footnotesize $1$};
\draw (10,-0.3) node {\footnotesize $1$};
\draw (10,2.7) node {\footnotesize $1$};
\draw (10.7,2.7) node {\footnotesize $1$};
\draw (12.3,2) node {\footnotesize $-1$};
\draw (9.3,2) node {\footnotesize $-1$};
\draw (11.5,0.2) node {\footnotesize $-1$};
\draw (8.5,3.75) node {\footnotesize $-1$};

\end{tikzpicture}
\caption{Pullbacks of two 2-torsion divisors from $K_4$.}
\label{Fig:Pullback}
\end{figure}
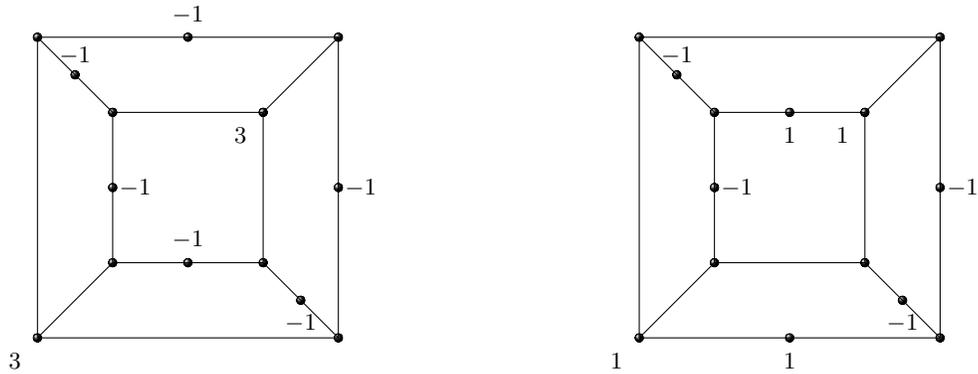
\end{example}

\begin{example}
\label{Ex:DilatedCovers}
By Lemma \ref{Lem:harmonicCovers}, for any cycle $\gamma \neq 0$ in the graph $K_4$ of the previous example, there is a unique unramified double cover $\varphi : \widetilde{K_4} \to K_4$ with dilation cycle $\gamma$.  These are depicted in Figures \ref{Fig:TriangleCover} and \ref{Fig:SquareCover}, with the dilation cycle marked by a solid line.  By Lemma \ref{Lem:Hurwitz}, each of these lifts to 8 distinct unramified double covers of the curve $X$.

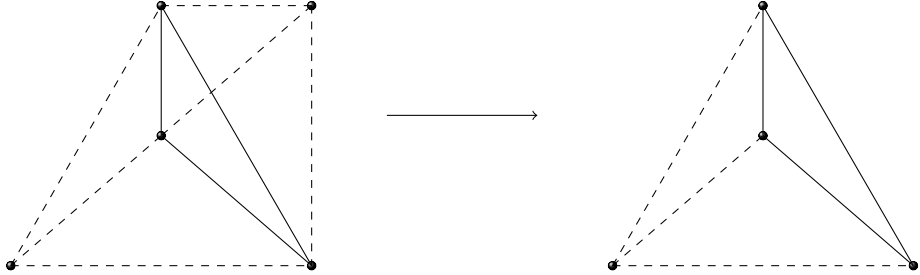
\begin{figure}[h]
\begin{tikzpicture}

\draw [ball color=black] (0,0) circle (0.55mm);
\draw [ball color=black] (4,0) circle (0.55mm);
\draw [ball color=black] (2,3.46) circle (0.55mm);
\draw [ball color=black] (2,1.73) circle (0.55mm);
\draw [ball color=black] (4,3.46) circle (0.55mm);
\draw [dashed] (0,0)--(4,0);
\draw [dashed] (0,0)--(2,3.46);
\draw [dashed] (0,0)--(2,1.73);
\draw [dashed] (4,3.46)--(4,0);
\draw [dashed] (4,3.46)--(2,3.46);
\draw [dashed] (4,3.46)--(2,1.73);
\draw (4,0)--(2,3.46);
\draw (4,0)--(2,1.73);
\draw (2,3.46)--(2,1.73);

\draw [->] (5,2)--(7,2);

\draw [ball color=black] (8,0) circle (0.55mm);
\draw [ball color=black] (12,0) circle (0.55mm);
\draw [ball color=black] (10,3.46) circle (0.55mm);
\draw [ball color=black] (10,1.73) circle (0.55mm);
\draw [dashed] (8,0)--(12,0);
\draw [dashed] (8,0)--(10,3.46);
\draw [dashed] (8,0)--(10,1.73);
\draw (12,0)--(10,3.46);
\draw (12,0)--(10,1.73);
\draw (10,3.46)--(10,1.73);

\end{tikzpicture}
\caption{The unique double cover with dilation cycle a given triangle in $K_4$.}
\label{Fig:TriangleCover}
\end{figure}

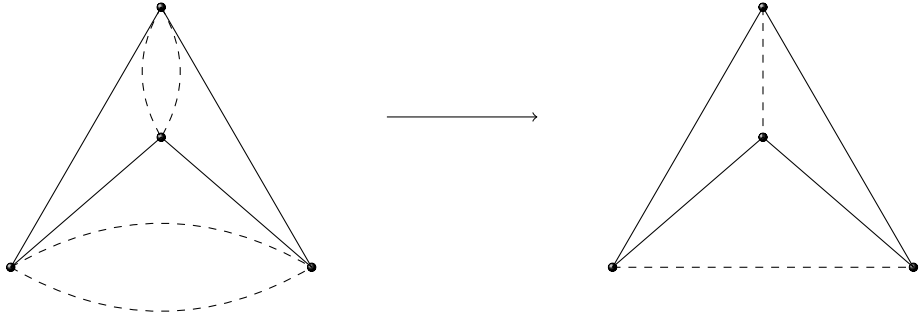
\begin{figure}[h]
\begin{tikzpicture}

\draw [ball color=black] (0,0) circle (0.55mm);
\draw [ball color=black] (4,0) circle (0.55mm);
\draw [ball color=black] (2,3.46) circle (0.55mm);
\draw [ball color=black] (2,1.73) circle (0.55mm);
\path [dashed] (0,0) edge [bend left] (4,0);
\path [dashed] (0,0) edge [bend right] (4,0);
\draw (0,0)--(2,3.46);
\draw (0,0)--(2,1.73);
\draw (4,0)--(2,3.46);
\draw (4,0)--(2,1.73);
\path [dashed] (2,3.46) edge [bend left] (2,1.73);
\path [dashed] (2,3.46) edge [bend right] (2,1.73);

\draw [->] (5,2)--(7,2);

\draw [ball color=black] (8,0) circle (0.55mm);
\draw [ball color=black] (12,0) circle (0.55mm);
\draw [ball color=black] (10,3.46) circle (0.55mm);
\draw [ball color=black] (10,1.73) circle (0.55mm);
\draw [dashed] (8,0)--(12,0);
\draw (8,0)--(10,3.46);
\draw (8,0)--(10,1.73);
\draw (12,0)--(10,3.46);
\draw (12,0)--(10,1.73);
\draw [dashed] (10,3.46)--(10,1.73);

\end{tikzpicture}
\caption{The unique double cover with dilation cycle a given square in $K_4$.}
\label{Fig:SquareCover}
\end{figure}

\end{example}

\section{Prym Varieties}
\label{Sec:Pryms}

We now study the kernel of the pushforward map on divisors.  Recall that the Jacobian of an augmented graph $\Gamma$ is the Jacobian of the unaugmented graph $\Gamma^\#$ obtained by attaching $g(v)$ loops of length $\epsilon>0$ at every point $v$. Those loops are referred to as \emph{virtual} loops.

Given an unramified double cover $\varphi : \widetilde{\Gamma} \to \Gamma$, consider the involution $\iota : \widetilde{\Gamma}^\# \to \widetilde{\Gamma}^\#$ that sends a point $x$ of $\widetilde{\Gamma}$ to the other preimage point of $\varphi (x)$, and sends a point on a virtual loop to the other point equidistant from the base of the same loop.  Extending by linearity, $\iota$ defines an involution $\iota : \Pic (\widetilde{\Gamma}) \to \Pic (\widetilde{\Gamma})$. Recall that by \cite[Theorem 3.4]{BakerFaber11}, the Jacobian of an anaugmented metric graph is isomorphic to $\Omega (\Gamma)^*/H_1(\Gamma,\ZZ)$, where $\Omega (\Gamma)^*$ is dual to the space harmonic $1$-forms on $\Gamma$. For an augmented graph $\widetilde{\Gamma}$, we define $\Omega(\widetilde{\Gamma})$ and $\Omega(\widetilde{\Gamma})^*$ to be the space of harmonic $1$-forms and its dual of $\Gamma^\#$.



\begin{proposition}
Let $\varphi : \widetilde{\Gamma} \to \Gamma$ be an unramified double cover.  If $\gamma (\varphi) = 0$, then the kernel of the pushforward map
\[
\varphi_* : \Jac (\widetilde{\Gamma}) \to \Jac (\Gamma)
\]
has two connected components, namely
\[
\ker \varphi_* = (\mathrm{Id} - \iota) \Pic^0 (\widetilde{\Gamma}) \cup (\mathrm{Id} - \iota) \Pic^1 (\widetilde{\Gamma}) .
\]
If $\gamma (\varphi) \neq 0$, it has only one connected component, given by
\[
(\mathrm{Id}-i)\Pic^0(\tilde{\Gamma}).
\]
\end{proposition}

\begin{proof}

Consider the natural pullback map on harmonic 1-forms (which we denote $\psi^*$ to avoid confusion with $\varphi^*$)
\[
\psi^* : \Omega (\Gamma) \to \Omega (\widetilde{\Gamma}) .
\]
It is easy to see that this map is injective.  Dualizing, we obtain a surjective map between the universal covers of the Jacobians
\[
\psi_* : \Omega (\widetilde{\Gamma})^* \to \Omega (\Gamma)^* .
\]
\noindent  We write $\Omega (\widetilde{\Gamma}/\Gamma)^*$ for the kernel of $\psi_*$.  The map $\psi_*$ descends to the map $\varphi_*$ of Jacobians by quotienting out $H_1 (\Gamma,\Z)$ and $H_1 (\widetilde\Gamma , \Z)$, respectively.  The kernel of $\varphi_*$ therefore corresponds to
$\psi_*^{-1} H_1 (\Gamma , \Z)$.  Since the map  $\psi_* \psi^*$ is multiplication by 2, it follows that
\[
\psi_*^{-1} H_1 (\Gamma , \Z) = \Omega (\widetilde{\Gamma}/\Gamma)^* + \frac{1}{2} \psi^* H_1 (\Gamma, \Z) .
\]

\noindent Passing to the quotient, we see that $\ker\varphi_*$ consists of translates of the $(g-1)$ dimensional torus
\[
P(\varphi) := \Omega (\widetilde{\Gamma}/\Gamma)^* / ( H_1 (\widetilde{\Gamma}, \Z) \cap \Omega (\widetilde{\Gamma}/\Gamma)^* ) \subset \Jac (\widetilde{\Gamma})
\]
by elements of $\varphi^* \Jac_2 (\Gamma)$.

To compute the number of connected components, it therefore suffices to determine when the pullbacks of $D,E \in \Jac_2 (\Gamma)$ lie in the same component.  Note that
\[
\varphi^* \Jac_2 (\Gamma) \subset \mathrm{im} (\mathrm{Id} - \iota) \subseteq \ker \varphi_* ,
\]
so if $D$ and $E$ lie in the same connected component of $\mathrm{im} (\mathrm{Id} - \iota)$, then they lie in the same connected component of $\ker \varphi_*$.

Let $D$ be a divisor of degree $k$ on $\widetilde{\Gamma}$.  Choose divisors $D_1 , D_2$ such that $D = D_1 + D_2$, $\deg(D_1) = \lceil \frac{k}{2} \rceil$, $\deg (D_2) = \lfloor \frac{k}{2} \rfloor$.  Then
\[
D - \iota (D) = (D_1 - \iota (D_2)) - \iota (D_1 - \iota (D_2)).
\]
Since $D_1 - \iota (D_2)$ has degree either 0 or 1, we see that
\[
\mathrm{im} (\mathrm{Id} - \iota) = (\mathrm{Id} - \iota) \Pic^0 (\widetilde{\Gamma}) \cup (\mathrm{Id} - \iota) \Pic^1 (\widetilde{\Gamma}) .
\]
It follows that $\ker \varphi_*$ has at most 2 connected components.

If $\gamma (\varphi) \neq 0$, then any point $x \in \varphi^{-1} (\gamma (\varphi))$ is fixed by the involution $\iota$.  Then, for any divisor $D \in \Pic^0 (\widetilde{\Gamma})$, we have
\[
D - \iota (D) = (D+x) - \iota (D+x),
\]
so
\[
(\mathrm{Id} - \iota) \Pic^0 (\widetilde{\Gamma}) = (\mathrm{Id} - \iota) \Pic^1 (\widetilde{\Gamma}),
\]
hence $\ker \varphi_*$ has only one connected component.

On the other hand, if $\gamma (\varphi) = 0$, then by Proposition \ref{Prop:Kernel}, $\varphi^*$ is injective, so $\vert \varphi^* \Jac_2 (\Gamma) \vert = 2^g$.  Since the $(g-1)$-dimensional torus $P(\varphi)$ contains only $2^{g-1}$ elements of order 2, we see that $\ker \varphi_*$ must contain at least two connected components.

\end{proof}

\begin{definition}
We define the \emph{Prym variety} $P(\varphi)$ to be the connected component of $\ker \varphi_*$ containing 0.  That is,
\[
P(\varphi) := (\mathrm{Id} - \iota) \Pic^0 (\widetilde{\Gamma}).
\]
\end{definition}

\noindent Note that the Prym variety $P(\varphi)$ is a real torus of dimension $g-1$.

Recall that in \cite{BakerRabinoff13}, it is shown that the Jacobian of the skeleton of a curve is the skeleton of its Jacobian. It is natural to ask whether an analogous  statement holds  for Pryms.
\begin{conjecture}
Let $f:C'\to C$ be an unramified double cover of smooth curves, and let $\phi:\Gamma'\to\Gamma$ be its tropicalization.  Then the skeleton of the Prym variety associated with $f$ is the tropical Prym variety associated with $\phi$. In other words, the Prym construction commutes with tropicalization.
\end{conjecture}

\noindent The Prym construction gives rise to a map from the Hurwitz space of admissible double covers $\mathcal{R}^g$ to the moduli space of Abelian varieties \cite{ABH}.
If the conjecture holds, then we obtain a commutative diagram
\[
\xymatrix{
\mathcal{R}_g \ar[r]^{\Trop} \ar[d]^{\text{Pr}}  &\mathcal{R}_g^{tr} \ar[d]^{\text{Pr}} \\
A_g \ar[r]^{\Trop} & A_g^{tr} }
\]
where $\mathcal{R}_g^{tr}$ is the skeleton of $\mathcal{R}_g$ (see \cite {CMR}), and $\text{Pr}$ is the map taking a double cover to the corresponding Prym variety (cf. \cite{Viviani}).

On an algebraic curve, there is an alternative way to define the Weil pairing using the algebraic version of the construction above.  Given a 2-torsion point $\cD$ in the Jacobian of an algebraic curve $X$, let $\overline{\varphi} : \widetilde{X} \to X$ be the corresponding unramified double cover.  Since $\overline{\varphi}_* \overline{\varphi}^*$ is multiplication by 2, the pullback of any 2-torsion point $\cE \in \Jac_2 (X)$ is in the kernel of $\overline{\varphi}_*$.  One can show that the Weil pairing $\lambda (\cD,\cE)$ is 0 if $\overline{\varphi}^* \cE$ is contained in the Prym variety, and 1 if it is contained in the other connected component \cite[Exercise B.23]{ACGH}.
 In this way, the Weil pairing is best understood as a pairing between $\Jac_2 (X)$ and the set of unramified double covers of $X$.

In the tropical setting, this construction only makes sense when the double cover has dilation cycle $\gamma (\varphi) = 0$.  For such covers, however, the specialization of the Weil pairing is well-behaved. Recall that $\Lambda$ denotes the kernel of  the tropicalization map $\Jac_2 (X) \to \Jac_2 (\Gamma)$.
\begin{proposition}
\label{Prop:WeilDescends}
Let $\cD \in \Lambda$, let $\overline{\varphi} : \widetilde{X} \to X$ be the corresponding double cover, and $\varphi : \widetilde{\Gamma} \to \Gamma$ the associated double cover of skeletons.  Then, for any $\cE \in \Jac_2 (X)$, we have $\lambda (\cD,\cE) = 0$ if and only if
\[
\varphi^* \Trop (\cE) \in P(\varphi).
\]
\end{proposition}

\begin{proof}
Since $\Lambda$ is isotropic for the Weil pairing, the linear map $\lambda (\cD, \cdot )$ factors through $\Jac_2 (\Gamma)$.  Exactly half the elements of $\Jac_2 (\Gamma)$ map to 0 and half to 1.  Similarly, the $(g-1)$ dimensional torus $P(\varphi)$ contains $2^{g-1}$ 2-torsion points.  Since $\varphi^*$ is injective, half the elements of $\varphi^* \Jac_2 (\Gamma)$ are contained in $P(\varphi)$ and half are contained in the other connected component.  By pigeonhole it therefore suffices to show that
\[
\Trop ((\mathrm{Id} - \iota) \Jac (\widetilde{X})) \subseteq (\mathrm{Id} - \iota) \Jac (\widetilde{\Gamma}).
\]
This follows from the commutativity of the following square.
\[
\xymatrix{
\Jac (\widetilde{X}) \ar[r]^{\Trop} \ar[d]^{\iota}  & \Jac (\widetilde{\Gamma}) \ar[d]^{\iota} \\
\Jac (\widetilde{X}) \ar[r]^{\Trop} & \Jac (\widetilde{\Gamma}) }
\]
\end{proof}

\begin{example}
\label{Ex:WeilPairing}
We now return to the example of the double cover of the complete graph $K_4$ by the cube $\widetilde{K_4}$.  By Proposition \ref{Prop:Lifting}, there exists a unique double cover $\overline{\varphi} : \widetilde{X} \to X$ specializing to this double cover of $K_4$.  We let $\cD \in \Jac_2 (X)$ be the 2-torsion divisor corresponding to the double cover $\overline{\varphi}$.

The divisor depicted in Figure \ref{Fig:NotPrym} is equivalent to the divisor pictured on the left in Figure \ref{Fig:Pullback}.  This divisor is easily seen to be in $(\mathrm{Id} - \iota) \Pic^1 (\widetilde{K_4})$.  In other words, it is not contained in the Prym variety $P(\varphi)$.  Conversely, the divisor depicted in Figure \ref{Fig:Prym} is equivalent to the divisor on the right in Figure \ref{Fig:Pullback}.  This divisor is easily seen to be in $(\mathrm{Id} - \iota) \Pic^0 (\widetilde{K_4})$.  In other words, it is contained in the Prym variety $P(\varphi)$.  By Proposition \ref{Prop:WeilDescends}, therefore, if $\cE \in \Jac_2 (X)$ specializes to $D_{\gamma}$, then $\lambda (\cD,\cE) = 1$ when $\gamma$ is a triangle, and $\lambda (\cD,\cE) = 0$ when $\gamma$ is a square.

\begin{figure}[h]
\begin{tikzpicture}

\draw [ball color=black] (0,0) circle (0.55mm);
\draw [ball color=black] (4,0) circle (0.55mm);
\draw [ball color=black] (0,4) circle (0.55mm);
\draw [ball color=black] (4,4) circle (0.55mm);
\draw [ball color=black] (1,1) circle (0.55mm);
\draw [ball color=black] (1,3) circle (0.55mm);
\draw [ball color=black] (3,1) circle (0.55mm);
\draw [ball color=black] (3,3) circle (0.55mm);
\draw [ball color=black] (2,4) circle (0.55mm);
\draw [ball color=black] (4,2) circle (0.55mm);
\draw [ball color=black] (1,2) circle (0.55mm);
\draw [ball color=black] (2,1) circle (0.55mm);
\draw [ball color=black] (0.5,3.5) circle (0.55mm);
\draw [ball color=black] (3.5,0.5) circle (0.55mm);
\draw (0,0)--(4,0);
\draw (0,0)--(0,4);
\draw (0,0)--(1,1);
\draw (4,0)--(4,4);
\draw (4,0)--(3,1);
\draw (0,4)--(4,4);
\draw (0,4)--(1,3);
\draw (4,4)--(3,3);
\draw (1,1)--(1,3);
\draw (1,1)--(3,1);
\draw (1,3)--(3,3);
\draw (3,1)--(3,3);
\draw (2,4.3) node {\footnotesize $1$};
\draw (4.3,2) node {\footnotesize $-1$};
\draw (1.3,2) node {\footnotesize $1$};
\draw (2,1.3) node {\footnotesize $-1$};
\draw (3.5,0.2) node {\footnotesize $1$};
\draw (0.5,3.75) node {\footnotesize $-1$};

\end{tikzpicture}
\caption{A divisor equivalent to the divisor on the left in Figure \ref{Fig:Pullback}.}
\label{Fig:NotPrym}
\end{figure}
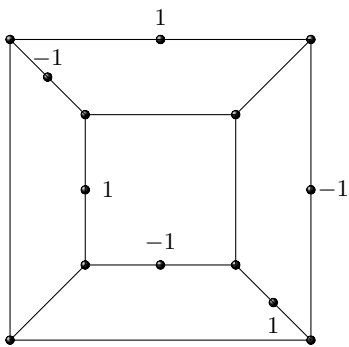

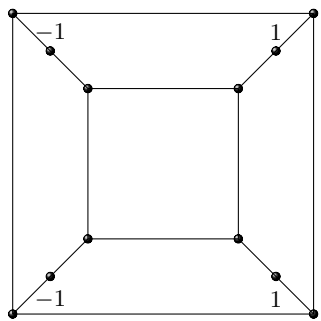
\begin{figure}[h]
\begin{tikzpicture}

\draw [ball color=black] (0,0) circle (0.55mm);
\draw [ball color=black] (4,0) circle (0.55mm);
\draw [ball color=black] (0,4) circle (0.55mm);
\draw [ball color=black] (4,4) circle (0.55mm);
\draw [ball color=black] (1,1) circle (0.55mm);
\draw [ball color=black] (1,3) circle (0.55mm);
\draw [ball color=black] (3,1) circle (0.55mm);
\draw [ball color=black] (3,3) circle (0.55mm);
\draw [ball color=black] (0.5,0.5) circle (0.55mm);
\draw [ball color=black] (0.5,3.5) circle (0.55mm);
\draw [ball color=black] (3.5,0.5) circle (0.55mm);
\draw [ball color=black] (3.5,3.5) circle (0.55mm);
\draw (0,0)--(4,0);
\draw (0,0)--(0,4);
\draw (0,0)--(1,1);
\draw (4,0)--(4,4);
\draw (4,0)--(3,1);
\draw (0,4)--(4,4);
\draw (0,4)--(1,3);
\draw (4,4)--(3,3);
\draw (1,1)--(1,3);
\draw (1,1)--(3,1);
\draw (1,3)--(3,3);
\draw (3,1)--(3,3);
\draw (0.5,0.2) node {\footnotesize $-1$};
\draw (3.5,0.2) node {\footnotesize $1$};
\draw (0.5,3.75) node {\footnotesize $-1$};
\draw (3.5,3.75) node {\footnotesize $1$};

\end{tikzpicture}
\caption{A divisor equivalent to the divisor on the right in Figure \ref{Fig:Pullback}.}
\label{Fig:Prym}
\end{figure}

\end{example}

\bibliography{math}

\end{document}